\documentclass{article}
\usepackage[T2A]{fontenc}
\usepackage{graphicx}
\usepackage[cp1251]{inputenc}
\usepackage[english,russian]{babel}
\usepackage[tbtags]{amsmath}
\usepackage{floatflt}
\usepackage{caption}
\usepackage{amsfonts,amssymb,mathrsfs,amscd}
\numberwithin{equation}{section}
\usepackage[hyper]{MSB-A3}
\theoremstyle{plain}
\newtheorem{theorem}{Theorem}

\theoremstyle{definition}
\newtheorem{definition}{Definition}
\newtheorem{proof}{Proof}
\newtheorem{remark}{Remark}[section]
\newtheorem{exmp}{Example}[section]
\newtheorem{Corollary}{Corollary}
\newcommand*{\QEDA}{\hfill\ensuremath{\blacksquare}}%
\newcommand*{\QEDB}{\hfill\ensuremath{\square}}%
\newcommand*{\DIAM}{\hfill\ensuremath{\Diamond}}
\begin{document}
\title{About one generalisation of the Leibniz theorem}

\author{Galina\,A.~Zverkina}
%\address{Moscow State University of Railway Engineering}
%\email{zverkina@gmail.com}

%\thanks{}

\date{01.04.2013}

\maketitle
\begin{abstract}
The well-known Leibniz theorem  (Leibniz Criterion or alternating series test) of convergence of alternating series is generalized for the case when the absolute value of terms of series are "not absolutely mo\-no\-to\-no\-us ly`` convergent to zero. Questions of accuracy of the estimation for the series  remainder are considered.
\end{abstract}

\hyphenation{re-mark mo-no-to-no-us re-pre-sents mo-no-to-no-us-ly }

\markright{Generalization of Leibniz theorem }

\footnotetext[0]{Author expresses gratitude to Professor V.N.Chubarikov (Department of mechanics and mathematics of Lomonosov Moscow State University) and Department of mathematics of the Yaroslavl State Technical University, the organizer of the International student's competition on mathematics in 2012.}

Leibniz theorem (Leibniz Criterion or alternating series test) give the pos\-si\-bi\-lity to demonstrate the convergence of an alternating series with decreasing to zero components. However in some cases values of series components  decrease to zero fluctuating. In some this case it can use the facts proved below.

\section{Generalization of Leibniz theorem}\label{s1}

\begin{definition}\label{d1}
The sequence $ \{a_n \} $ is called $Z (\omega) $-monotonously increasing (de\-cre\-a\-sing) on set $ \mathfrak {D} $ ($ \omega \in {\mathbb {N}} $) if $ \forall k \in\mathfrak {D} $ it is carried out $a _ {k +\omega} \geqslant a_k $ (accordingly
$a _ {k +\omega} \leqslant a_k $).\DIAM
\end{definition}

\begin{theorem}\label{t1}
If the sequence $a_n $ is $Z (2\omega-1) $-monotonously decreasing for $n\geqslant n_0$ $ (\omega, \, n\in\mathbb {N}) $ and $ \lim \limits _ {n \rightarrow + \infty} a_n=0$ then a series $ \sum\limits _ {n=n_0} ^ {\infty} (-1) ^n a_n $ converges. Thus the series remainder or the difference between the sum of the series $S =\sum\limits _ {n=n_0} ^ {\infty} (-1) ^n a_n $ and its partial sum $S_m =\sum\limits _ {n=n_0} ^m (-1) ^n a_n $ can be estimated as follows:

 \begin{equation}\label{RGZ}
    \begin{array}{c}
      R_m=S-S_m, \\
      |R_m|\leqslant\sum\limits_{n=m+1}^{m+2\omega }  |a_k|.
    \end{array}
 \end{equation}
 \nopagebreak\QEDB
\end{theorem}

\begin{remark}
It is easy to see, in the case when the sequence $ \{a_n \} $ is $Z (2\omega) $-mo\-no\-to\-no\-us and $ \lim \limits _ {n \rightarrow + \infty} a_n=0$, the series $ \sum\limits _ {k=n_0} ^ {\infty} (-1) ^n a_n $ can be not convergent. \DIAM
\end{remark}

\begin{exmp}
$$
 a_n = \left\{
 \begin{array}{rcl}
 \displaystyle \frac{1}{k^2} &\mbox{if}& n=2k,\\
 \mbox{ }&\mbox{ }&\mbox{ }\\
\displaystyle \frac{1}{k}& \mbox{if}& n=2k-1,\\
 \end{array}
 \;k\in\mathbb{N}.
 \right.
 $$
 Obviously, the sequence $ \{a_n \} $ is $Z (2) $-monotonous, and a series $ \sum\limits _ {n=1} ^ {\infty} (-1) ^n a_n $ represents a difference of a harmonic series (divergent infinite)  and a converging   series $\sum\limits_{n=1}^\infty\frac1{n^2}$:\\
$\sum\limits_{n=1}^\infty(-1)^na_n=-\frac{1}{1^2}+\frac{1}{1}-\frac{1}{2^2}+\frac{1}{2}-\frac{1}{3^2}+\frac{1}{3}-\ldots=+\infty$.\QEDA
\end{exmp}

If $\omega =1$ $\;$ then $Z(2\omega -1)$-monotony turns to usual monotony  $(0\leqslant a_{n+1}\leqslant a_n)$, and the Theorem 1 turns to the well-known Leibniz Theorem about alternating series:

\begin{theorem}[(G.W.\,von\,Leibniz, 1682)]\label{t2}
An alternating series \\ $S =\sum\limits _ {n=1} ^ {\infty} (-1) ^ {n+1} b_n $ converges, if both conditions are satisfied:

1. $\forall n \;\;\; b_{n}\geqslant b_{n+1}\geqslant0$;

2. $\lim\limits_{n \rightarrow +\infty} b_n=0$.

Besides, the partial sum of the series satisfies to an inequality:

$$
0\leqslant\sum\limits_{n=1}^{\infty}(-1)^{n+1}b_n\leqslant b_1.
$$
\nopagebreak\QEDB
\end{theorem}

The corollary from the Leibniz theorem  allows to estimate an error of cal\-cu\-la\-tion of the partial sum of a series $S_m =\sum\limits _ {n=1} ^ {m} b_n $.

\begin{Corollary}
The remainder $R_m=S-S_m$ of  a convergent alternating series satisfies  an inequality:
\begin{equation}\label{ocleib}
    |R_m|\leqslant b_{m+1}.
\end{equation}

Further this estimation we will denote $R_m^L $: \linebreak $ |R_m |\leqslant R_m^L $.

Moreover, it is possible to approve following equality:
\begin{equation}\label{LeibOcenka}
    \begin{array}{c}
      R_m=\theta\cdot b_{m+1}, \\ \\
      0\leqslant\theta\leqslant 1.
    \end{array}
\end{equation}

$\theta$ can be equal to 0 and 1, for example, for the series $\displaystyle\sum\limits_{n=1}^{\infty}\frac{(-1)^{n+1}}{\left[\frac{n+1}{2}\right]} $. \QEDB
\end{Corollary}

Let $Z $-\emph{series} be  a series satisfying to conditions of Theorem 1, and $L $-\emph{series} --- a series satisfying to conditions of Theorem 2.

It's well-known, the Leibniz theorem is a special case of the Dirichlet's theorem (Dirichlet test):

\begin{theorem}[(J.P.G.\,Lejeune\,Dirichlet) ]\label{dir}

If
$$
\forall N\in\mathbb{N}\;\;\; \left| \sum \limits_{n=1}^N b_n\right|<M,
$$
where M is some constant, and
$$
\forall n\in\mathbb{N}\;\;\; a_{n}\geqslant a_{n+1},\; \lim \limits_{n\rightarrow\infty}a_n=0,
$$
then the series
 $\sum \limits_{n=1}^\infty a_n b_n$ converges. \QEDB
\end{theorem}

Below (examples (\ref {3-z}), (\ref {zv-1}), (\ref {5-z})) we give samples of series for which the Theorem 1 allows to prove convergence, but Dirichlet's Theorem is inapplicable or its application involves the big technical difficulties.

\begin{proof} [Theorem 1]
Let $a_n $ be a $Z (2\omega-1)  $-monotonously decreasing sequence converging to 0 for $i\geqslant n_0$. For simplicity we will consider $n_0=1$.

Let's consider series $ \sigma _k =\sum\limits _ {j=1} ^ {\infty} \alpha _ {j, k} $, ($k=1,2, \ldots, 2\omega-1$), where
\begin {equation}
    \alpha _ {j, k} = \left \{\begin {array} {ccc}
                           (-1) ^ {j} a_j & \mbox {for} & j=m\cdot (2\omega-1) +k, \\ \\
                           0 & \mbox {for} & j\neq m\cdot (2\omega-1) +k,
                         \end {array}
    \right. \\
    m\in\mathbb {Z}.
\end {equation}

Then
$$
\begin{array}{c}
  \sigma _k=\sum\limits_{j=1}^{\infty}\alpha _{j,k}=\\ \\
  =\underbrace{0+0+\ldots +0}_{k-1 \mbox{ terms}}+(-1)^k a_k + \underbrace{0+0+\ldots +0}_{2\omega -2 \mbox{ terms}}+(-1)^{k+2\omega -1} a_{k+2\omega -1}+ \\
  \\
  +\underbrace{0+0+\ldots +0}_{2\omega -2 \mbox{ terms}}+(-1)^{k+2(2\omega -1)} a_{k+2(2\omega -1)}+\ldots\\ \\ \ldots+(-1)^{k+m(2\omega -1)} a_{k+m(2\omega -1)}
  +\underbrace{0+0+\ldots +0}_{2\omega -2 \mbox{ terms}}+\\ \\+(-1)^{k+(m+1)(2\omega -1)} a_{k+(m+1)(2\omega -1)}+\ldots
\end{array}
$$

Actually a series $ \sigma _k $ is a series $\widetilde{\sigma } _k$ rarefied by many  zeros;
$$
\widetilde{\sigma } _k=\sum\limits_{m=0}^{\infty}(-1)^{k+m(2\omega -1)} a_{k+m(2\omega -1)}=(-1)^k \sum\limits_{m=0}^{\infty}(-1)^m a_{k+m(2\omega -1)}.
$$

A series $ \widetilde {\sigma} _k $ satisfies the conditions of the Leibniz theorem, as from \linebreak $Z (2\omega-1) $-monotonicity of sequence $a_i $ usual monotonicity of sequence\linebreak $ a _ {k+m (2\omega-1)} $ is followed;
the sum of this series is finite, and the remainder
$$
\begin{array}{c}
  \widetilde{\rho}_{p,k}=\sigma_k-\sigma_{p,k} = (-1)^{k}\sum\limits_{m=p+1}^{\infty}(-1)^{m} a_{k+m(2\omega -1)},\\ \\
  \mbox{where }\\ \\\sigma_{p,k}=\sum\limits_{m=0}^{p}(-1)^{k+m(2\omega -1)} a_{k+m(2\omega -1)}=(-1)^{k}\sum\limits_{m=0}^{p}(-1)^{m} a_{k+m(2\omega -1)}.\\

\end{array}$$
is estimated by the first rejected term according to (\ref {LeibOcenka}):
$$
\begin{array}{c}
  |\widetilde{\rho}_{p,k}|\leqslant | a_{k+(p+1)(2\omega -1)}| \\ \\
\mbox{or}\\ \\
\widetilde{\rho}_{p,k}=\theta_{p,k} \cdot (-1)^{k+p+1}a_{k+(p+1)(2\omega -1)},\\ \\
0\leqslant\theta_{p,k}\leqslant 1.
\end{array}
$$

From this we can see, that the series $ \sigma _k $ also converges to the finite sum, and its remainder is estimated by the first  \emph{nonzero}  rejected component:

$$
\rho _{q,k}=\sum\limits_{j=1}^{\infty}\alpha _{j,k}-\sum\limits_{j=1}^{q}\alpha _{j,k}=\theta_{q,k} \cdot (-1)^{k+(p+1)(2\omega -1)}a_{k+(p+1)},
$$
when $k+p(2\omega -1)\leqslant q<k+(p+1)(2\omega -1)$, $0\leqslant\theta_{q,k}\leqslant 1$.

Otherwise it can be written down so:

$$
\rho _{q,k}=\theta_{q,k} \cdot \sum\limits_{j=q+1}^{q+(2\omega -1)} \alpha_{j,k},
$$
as in the sum $ \sum\limits _ {j=q+1} ^ {q + (2\omega-1)} \alpha _ {j, k} $ there are only one nonzero component, this is \\$ (-1) ^ {k + (q+1) (2\omega-1)} a _ {k + (q+1) (2\omega-1)}
$, where $k+p (2\omega-1) \leqslant q <k + (p+1) (2\omega-1) $.

It is easy to see, that the initial series $ \sum\limits _ {n=n_0} ^ {\infty} (-1) ^n a_n $ is the sum of the series $ \sum\limits _ {j=1} ^ {\infty} \alpha _ {j, k} $, $1 \leqslant k\leqslant 2\omega-1$, accordingly the partial sum
$S_q =\sum\limits _ {k=1} ^ {2\omega-1} \sigma _ {q, k} $. From the implies existence of finite limits $ \lim\limits _ {q\rightarrow \infty} \sigma _ {q, k} =s_k $ existence of a finite limit of $S_q$ follows;
$ \lim\limits _ {q\rightarrow \infty} S_q=S =\sum\limits _ {k=1} ^ {2\omega-1} s_k $, thus\\
$ |S-S_q |\leqslant \sum\limits _ {k=1} ^ {2\omega-1} |s _ {k}-\sigma _ {q, k} | \leqslant\sum\limits _ {n=q+1} ^ {q+2\omega-1} a_n $ --- the sum of absolute values of estimates (\ref{ocleib}) of the remainders of series $ \sigma_k $.

As signs of the remainders of series $ \sigma_k $ alternate, last estimation can be improved as follows:

\begin{equation}\label{ocenka}
    |S-S_q|\leqslant\max\left\{ \sum\limits_{r=1}^{\omega -1}a_{q+2r},\sum\limits_{r=1}^{\omega }a_{q+2r-1}\right\}.
\end{equation}

Further the estimation $ \max\left \{\sum\limits _ {r=1} ^ {\omega-1} a _ {m+2r}, \sum\limits _ {r=1} ^ {\omega} a _ {m+2r-1} \right \} $ we will denote $R_m^Z $: \\ $R_m^Z\geqslant |S-S_m | = | R_m | $.

And it is more exact:

\begin {equation} \label {RMZ}
        S-S_m =\sum\limits _ {j=1} ^ {2\omega-1} (-1) ^ {m+j} \tilde \theta _ {m+j} a _ {m+j}, \; 0 \leqslant\tilde \theta_r \leqslant 1.
\end {equation}
\QEDA
\end {proof}

Generally the estimation (\ref {ocenka}) is not improved asymptotically, that it is possible to see from the following example.
\begin{exmp}\label{3-z}
Let the sequence $a_n $ be defined as follows:
\begin{equation}\label{0}
    a_n=\left\{\begin{array}{ccc}
                \frac{1}{k}+\frac{1}{2^k}, & \mbox{if} & n=3(2k-1)-2; \\ \\
                \frac{1}{10^{k}}, &  \mbox{if}  &  n=3(2k-1)-1;\\ \\
                \frac{1}{k}+\frac{1}{2^k}, &  \mbox{if} &  n=3(2k-1);\\ \\
                \frac{1}{k},  &  \mbox{if}  &  n=3\cdot 2k-2;\\ \\
                \frac{1}{10^{k}}, &  \mbox{if}  &  n=3\cdot 2k-1;\\ \\
                 \frac{1}{k}, & \mbox{if}  &  n=3\cdot 2k,
              \end{array}
              \right.
              \; k\in\mathbb{N}.
\end{equation}

It is easy to see, that the sequence $ \{a_n  \} $ converges to 0, being $Z (3) $-mo\-no\-to\-no\-us, and a series $ \sum\limits _ {n=1} ^ {\infty} (-1) ^ {n} a_n $ is not $L $-series. The sum of this $Z(3)$-series is

$$
\begin{array}{c}
  \sum\limits_{n=1}^\infty (-1)^{n+1}a_n= \\ \\
  = \underbrace{\frac{1}{1}+\boxed{\frac{1}{2}}}_{a_1}-\underbrace{{\frac{1}{10}}}_{a_2}+\underbrace{\frac{1}{1}+\boxed{\frac{1}{2}}}_{a_3}-\underbrace{\frac{1}{1}}_{a_4}
  +\underbrace{{\frac{1}{10}}}_{a_5}-\underbrace{\frac{1}{1}}_{a_6}+\\
  \\
  +\underbrace{\frac{1}{2}+\boxed{\frac{1}{2^2}}}_{a_7}-\underbrace{{\frac{1}{10^{2}}}}_{a_8}+\underbrace{\frac{1}{2}+\boxed{\frac{1}{2^2}}}_{a_9}-\underbrace{\frac{1}{2}}_{a_{10}}
  +\underbrace{{\frac{1}{10^{2}}}}_{a_{11}}-\underbrace{\frac{1}{2}}_{a_{12}}+ \\
  \\
   +\underbrace{\frac{1}{3}+\boxed{\frac{1}{2^3}}}_{a_{13}}-\underbrace{{\frac{1}{10^{3}}}}_{a_{14}} +\underbrace{\frac{1}{3}+\boxed{\frac{1}{2^3}}}_{a_{15}}
   -\underbrace{\frac{1}{3}}_{a_{16}}+\underbrace{{\frac{1}{10^{3}}}}_{a_{17}}-\underbrace{\frac{1}{3}}_{a_{18}}+\ldots\\ \\
   \ldots +\underbrace{\frac{1}{k}+\boxed{\frac{1}{2^k}}}_{a_{6k-5}}-\underbrace{{\frac{1}{10^{k}}}}_{a_{6k-4}} +\underbrace{\frac{1}{k}+\boxed{\frac{1}{2^k}}}_{a_{6k-3}}
   -\underbrace{\frac{1}{k}}_{a_{6k-2}}+\underbrace{{\frac{1}{10^{k}}}}_{a_{6k-1}}-\underbrace{\frac{1}{k}}_{a_{6k}}+\ldots\\

\end{array}
$$
converges to the sum $S=2$ (all components are reduced except boxed ones). Thus

$$
\begin{array}{c}
  S_{6k}=2\left(1-\frac{1}{2^k}\right);\\ \\
   S_{6k+1}=2\left(1-\frac{1}{2^k}\right)+\frac{1}{k}+\frac{1}{2^k};\\ \\
   S_{6k+2}=2\left(1-\frac{1}{2^k}\right)+\frac{1}{k}+\frac{1}{2^k}-\frac{1}{10^k};\\ \\
   S_{6k+3}=2\left(1-\frac{1}{2^{k+1}}\right)- \frac{1}{10^k}+\frac{2}{k};\\ \\
   S_{6k+4}=2\left(1-\frac{1}{2^{k+1}}\right)+\frac{1}{k}- \frac{1}{10^k};\\ \\
  S_{6k+5}=2\left(1-\frac{1}{2^{k+1}}\right)+\frac{1}{k}.
\end{array}
$$

That is
$$
R_{6k+3}=\sum\limits_{n=1}^\infty (-1)^{n+1}a_n - \sum\limits_{n=1}^{6k+3} (-1)^{n+1}a_n = \frac{1}{2^{k}} +\frac{1}{10^{k}} -\frac{2}{2k},
$$
i.e. when $k\rightarrow\infty\;$ $R_{6k+3}\sim \frac{2}{k} \sim \left(a_{6k+4}+a_{6k+6}\right)$.

However $R _ {6k} = \frac {1} {2 ^ {k-1}} $, that is the absolute value of the remainder of a series has big fluctuations.

The series (\ref{0}) presented here, apparently, cannot be easily investigated by the Dirichlet test (Theorem \ref {dir}). \QEDA
\end{exmp}

\section{Conditions of applicability of the Theorem  \ref{t1}}\label{s2}

Often components of a numerical series represent values of some continuous function in integer points: $a_n=f (n) $. Therefore for research of convergence of series $ \sum\limits _ {n=1} ^ \infty (-1) ^n f (n) $ in the
 case when $f (x) $ is not a monotonous function, it is natural to extend the concept of $Z $-monotony to all functions.

\begin{definition}\label{Zv}
Function $f (x) $ is called $Z (T) $-monotonously increasing (de\-cre\-a\-sing) on the set $ \mathfrak {D} $ ($T> 0$) if $ \forall x \in\mathfrak {D} $ it is carried out $f (x+T) \geqslant f (x) $ (accordingly $f (x+T) \leqslant f (x) $). \DIAM
\end{definition}

%\begin{floatingfigure}{51mm}
%\noindent
%\hfil
%\includegraphics[width=50mm]{Zvup.eps}
%\hfil
%\caption{{$Zv $-monotonously increasing function with parameter 10}}
%\label{up}
%\hfil
%\end{floatingfigure}

However the fact, that $f (x) $ is $Z (T) $-monotonous function, does not allow to draw a conclusion that the sequence $f (n) $ is $Z (k) $-monotonous.
Indeed, the function $ \varphi (x) = \mbox {ln} x+x \sin ^2 x $ is $Z (2\pi) $-monotonously increasing for $x> 0$, however at any natural $k $ it is not $Z (k) $-monotonous. Therefore it is necessary to introduce the concept of  the strong (or very) $Z $-monotony.

\begin{definition}
Function $f (x) $ is called \,$Zv$-monotonously\footnote {$Z $-$ very $-monotonously}  increasing (de\-cre\-a\-sing) on set $ \mathfrak {D} $, if
$$\exists T> 0: \forall x\in\mathfrak {D}, \forall \tau> 0 \; \; f (x+T +\tau) \geqslant f (x) $$
(accordingly $f (x+T +\tau) \leqslant f (x) $). \DIAM
\end{definition}

\begin{figure}\center
\includegraphics[width=3.25in]{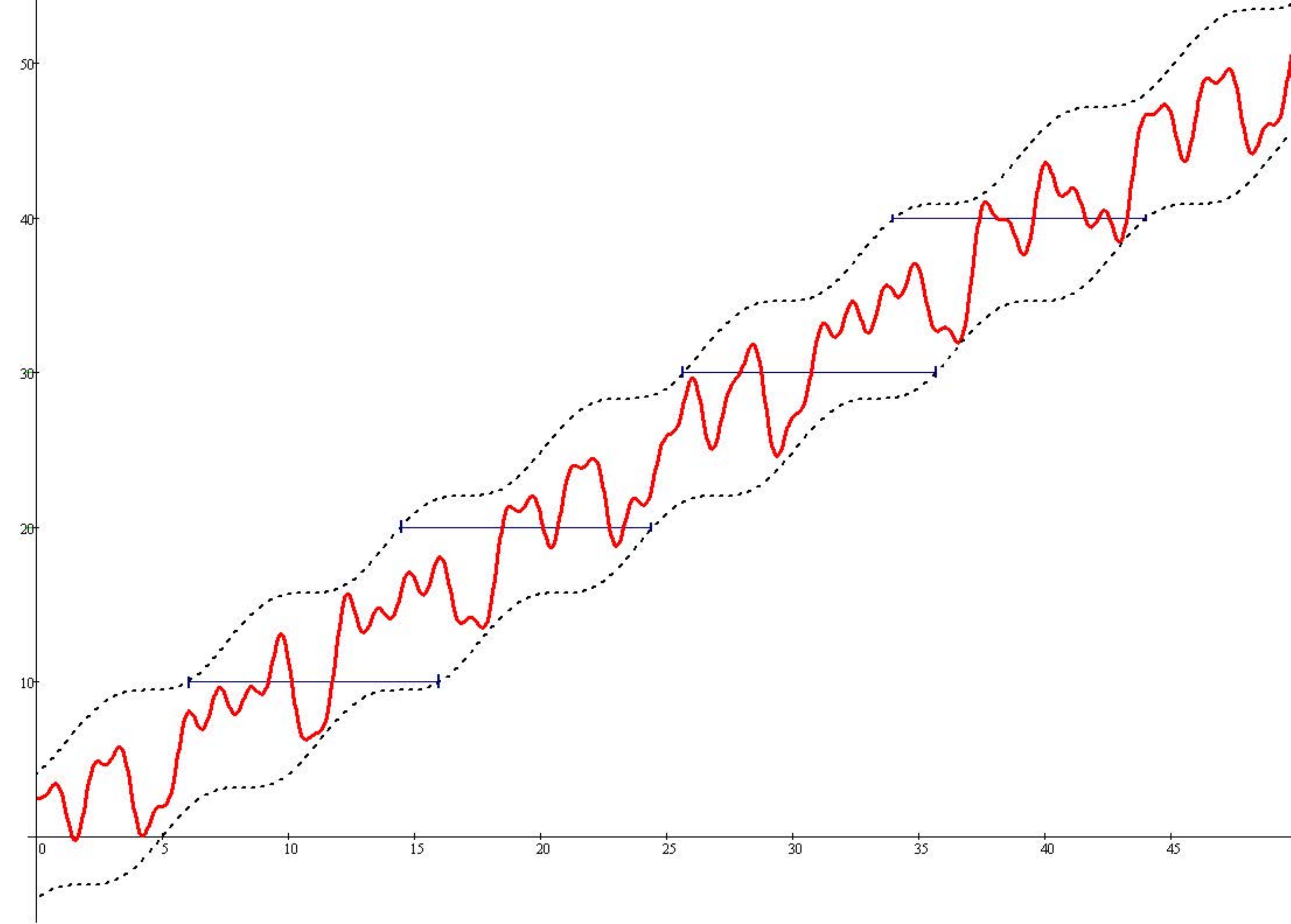}
\caption{{$Zv $-monotonously increasing function with parameter 10}}
\label{up}
\end{figure}

I.e. $f (x) $ is $Zv$-monotonous, if it is $Z (T +\tau) $-monotonous for some fixed $T> 0$ and any $ \tau> 0$.
Let's introduce the parameter of $Zv $-monotonous increasing function on set $ \mathfrak {D} $:

$Par_{Zv}(f(x))=\inf \{T>0: \forall\tau>0\;\; \forall x\in\mathfrak{ D}\;  f(x+T+\tau)\geqslant f(x)\}$.

The parameter of a $Zv $-monotonously decreasing function is similarly defined.

If the parameter of $Zv$-monotonous function  is equal to 0 this function is monotonous in usual sense.

\begin{figure}\center
\includegraphics[width=4.25in]{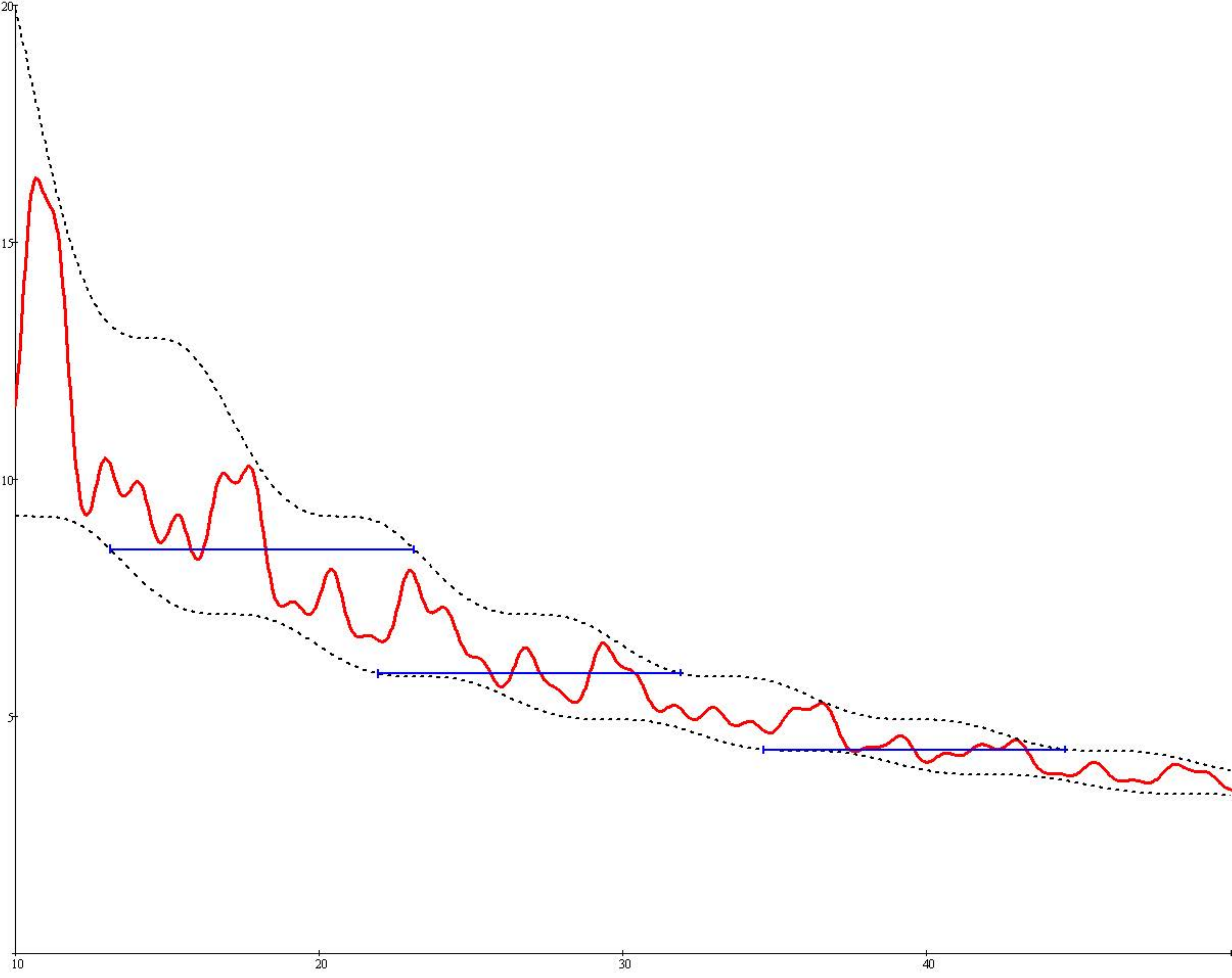}
\caption{{$Zv $-monotonously decreasing function with parameter 10} }\label{down}
\end{figure}

 Definition \ref {Zv} implies, that for any $Zv $-monotonous functions $f (x) $ on a set $ \mathfrak {D} $ we can find such monotonous\footnote {Non-strict monotony means: $ \varphi (x) $ it is monotonous on
 $ \mathfrak {D} $, if \\$ \forall a <b \in\mathfrak {D} \; \; \varphi (a) \leqslant \varphi (b) $ or $ \forall a <b \in\mathfrak {D} \; \; \varphi (a) \geqslant \varphi (b) $.} function $ \varphi (x) $, that $ \forall t\in\mathfrak {D} $ the value $f (t) $ is between numbers $ \varphi (t) $ and $ \varphi (t+T) $, where $T\geqslant Par _ {Zv} (f (x)) $, that is the graph of the function $f (x) $ lays in a strip between two   graph of monotonous functions; the width across this strip is limited, but, naturally, it is not less then the parameter of $Zv $-monotonous function (see Fig. \ref {up} and \ref {down}). However to find such function $ \varphi (x) $ is not always simply.
Therefore for the proof of $Zv $-monotonous increase of function $f (x) $ it is enough to find two monotonous functions $ \varphi_1 (x) $, $ \varphi_2 (x) $, such, that $ \varphi_1 (x) \leqslant f (x) \leqslant\varphi_2 (x) $ and $ \exists T> 0: \forall x\in\mathfrak {D} \; \varphi_1 (x+T)> \varphi_2 (x) $ (In this case $T\geqslant Par _ {Zv} (f (x)) $. The question about $Zv $-monotonous decrease is similarly solved.

%\begin{floatingfigure}[l]{51mm}
%\noindent
%\hfil
%\includegraphics[width=50mm]{Zvdown.eps}
%\hfil
%\caption{{$Zv $-monotonously decreasing function with parameter 10} }\label{down}
%\sloppy
%\hfil
%\end{floatingfigure}

In most cases it is difficult to define parameter $Zv $-monotonous function, but it is possible to receive an estimation of this parameter from above.

\begin{exmp}\label{cos}
Let $0 <\alpha\leqslant 1$ and the function $p (x) $  is bounded: $ |p (x) | <M $. We will show, that the function $ f (x) =x ^\alpha+p (x) x ^ {\alpha-1} $ is $Zv $-monotonously increasing for $x> \frac {(1-\alpha) M} {\alpha} $.

Let's consider functions $q (x) = x ^\alpha +\frac {M} {x ^ {1-\alpha}} $ and $r (x) =x ^\alpha-\frac {M} {x ^ {1-\alpha}} $. These functions monotonously increase when $x> \frac {(1-\alpha) M} {\alpha} $. Obviously, \linebreak $r (x) \leqslant F (x) \leqslant q (x) $ (the graph of the functions $F (x) $ is in a strip between graph of functions $q (x) $ and $r (x) $).

Let's show, that the distance across between graph of functions $q (x) $ and $r (x) $ is limited if $x $ is big enough.

Let's choose a point $x_0$ in which function $q (x) $ increases: it is carried out if $x_0> \frac {(1-\alpha) M} {\alpha} $. We will find a point $x_1: \; r (x_1) =q (x_0) $. We will draw a tangent line to the graph of the function $q (x) $ at a point $C $ with coordinates $ (x_1; q (x_1)) $ (Fig. \ref {rad}) and a horizontal straight line through a point $B (x_1; r (x_1)) $ before crossing a tangent line $AC $. We will estimate size $ |bB | = x_1-x_0$ --- distance between the graphs of functions $q (x) $ and $r (x) $ across.
\begin{equation}
\begin{array}{c}
  |bB|<|AB|=|BC| \cot \angle CAB;\\
  \\
  |BC|=\frac{2M}{x^{1-\alpha}}; \\ \\
 \tan\angle CAB=q'(x_1)=\alpha x_1^{\alpha-1}-(1-\alpha)M x_1^{\alpha-2};\\ \\
|bB|<|AB|=\frac{|BC|}{\tan \angle CAB}=\frac{2M}{\alpha-\frac{M(1-\alpha)}{x_1^2}}.
\end{array}
\end{equation}
Function on the right side decreases if $x $ ($x> 0$) increases, therefore for all $x> x_0$ the distance across between graphs of functions $q (x) $ and $r (x) $ will be less than $ T (x_0) = \frac {2M} {\alpha-\frac {M (1-\alpha)} {x_1^2}} $, that is $r (x)> q (x-T (x_0)) $; hence, $F (x) $ $Zv $-monotonously increases at $x> x_0$ with parameter smaller than $T (x_0) $. \QEDB
\end{exmp}

\begin {theorem}
If $f (x)$ $Zv $-monotonously decreases and $ \lim\limits _ {x\rightarrow\infty} f (x) =0$ then a series $ \sum\limits _ {n=n_0} ^ {\infty} (-1) ^n f (n) $ converges. \QEDB
\end {theorem}

\begin {proof}
Let's find some odd number $2 \omega-1\geqslant Par _ {Zv} (f (x)) $. The sequence $ \left \{f (n) \right \} $ is $Z (2\omega-1) $-monotonous. Therefore a series $ \sum\limits _ {n=n_0} ^ {\infty} (-1) ^n f (n) $ converges. \QEDA
\end {proof}

\begin{figure}\center
\includegraphics[width=2.25in]{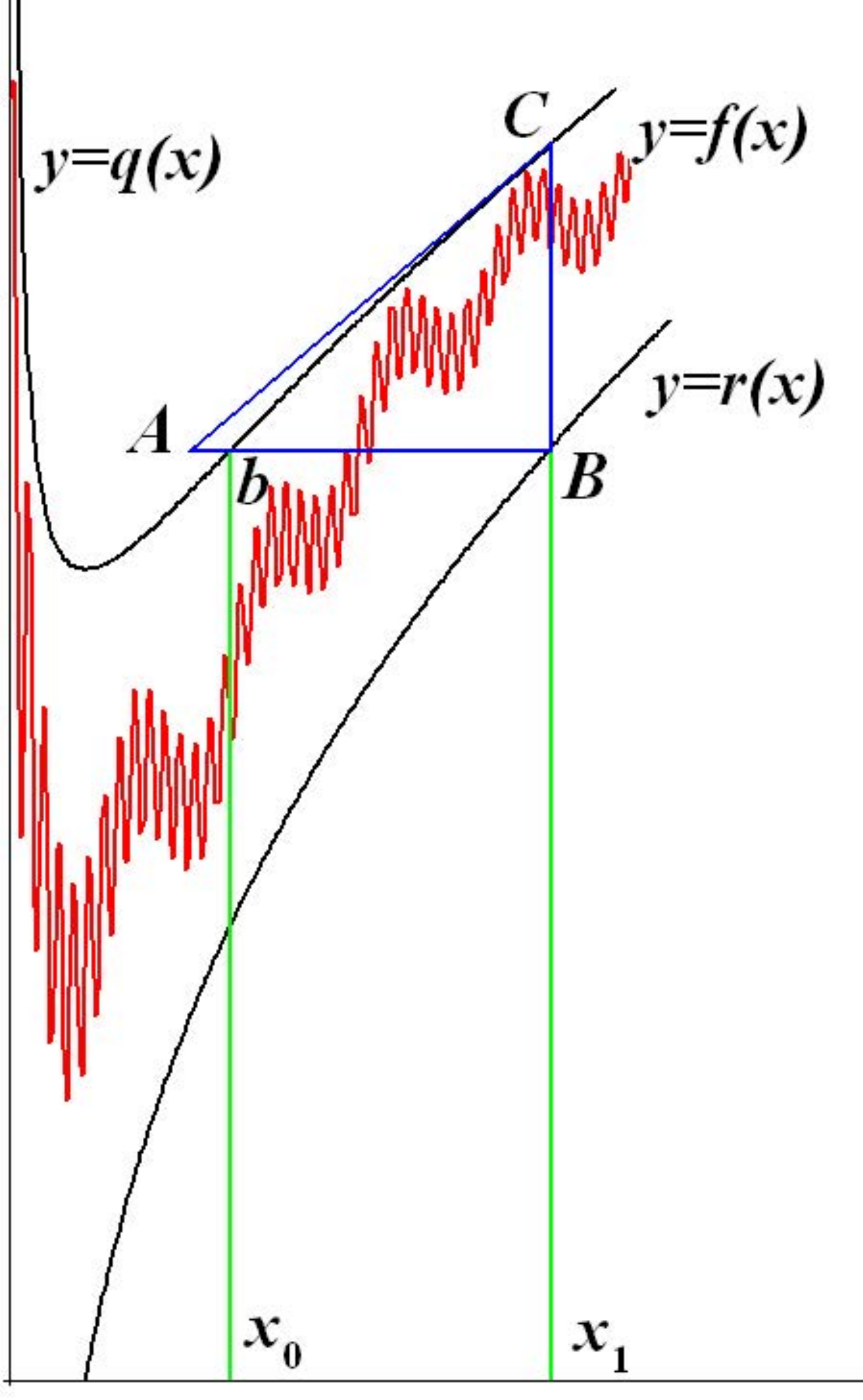}
\caption{Illustration for Example \ref{cos}}\label{rad}
\end{figure}

\begin {exmp} \label {zv-1}
The series $ \sum\limits _ {n=1} ^ {\infty} \frac {(-1) ^n n ^\beta} {n+p (n)} $ converges, if 0$ \leqslant\beta <1$ and the function $p (x)$  is bounded. It follows from the fact, that function $g (x) = \frac {x ^\beta} {x+p (x)} $ $Zv $-mo\-no\-to\-no\-us\-ly decreases to 0, as $g (x) = \frac {1} {f (x)} $, where $f (x) $ is considered in an example \ref {cos} $Zv $-mo\-no\-to\-no\-us\-ly function increasing to infinity (here $ \beta=1-\alpha $). \QEDB
\end {exmp}

\begin {exmp} \label {5-z}
It is easy to notice, that the function $g (x) = \frac {1} {x+2\cos x} $ is $Zv $-mo\-no\-to\-no\-us and
\begin {equation} \label {Jar}
    Par _ {Zv} (g (x)) \leqslant 2\pi.
\end {equation}

 That is a series $ \sum\limits _ {n=1} ^ {\infty} \frac {(-1) ^ {(n-1)}} {n+2\cos n} $ is $Z $-series\footnote{At VI International student's competition on the mathematics of 2012 in Yaroslavl organizers have suggested to investigate convergence of this series. It is possible to prove it convergence using some trigonometrical transformations, however little changes of the formula make it impossible to use the solution of a problem in this way (offered by organizers of competitions).
 Reflections over this problem have led the author to a writing of present article.}; sequence \linebreak $ \left \{a_n\right \} =\left \{\frac {1} {n+2\cos n} \right \} $ \; $Z (7) $-monotonously decrease to zero ($7> 2\pi $). It means, the series remainder $R_m =\sum\limits _ {n=m+1} ^ {\infty} \frac {(-1) ^ {(n-1)}} {n+2\cos n} $ can be estimated by the sum of four terms: $ |R_m | \leqslant a _ {m+1} +a _ {m+3} +a _ {m+5} +a _ {m+7} $.
However the estimation (\ref {Jar}) can be improved according to Example \ref {cos} reasonings.

Indeed, $x-2\leqslant\frac {1} {g (x)} =x+2\cos x\leqslant x+2$, (see Example \ref {cos} with
$M=2$, $ \alpha=1$, $ \tan \angle CAB=1$, $ |bB | = | AB | = 4$).
Hence,  $Par _ {Zv} (g (x)) \leqslant 4$, and the sequence
$ \left \{a_n\right \} $ is
$Z (5) $-monotonously decreasing. It gives the best estimation for
$R_m $: $ |R_m | \leqslant a _ {m+1} +a _ {n+m} +a _ {m+5} $. However it is possible to see, when $m $ is big enough then
$ |R_m | \leqslant a _ {m+1} $. I.e. estimations (\ref {LeibOcenka}) and (\ref {ocenka}) in some cases can be improved. \QEDB
\end {exmp}

\section {On the accuracy of the estimation of the re\-mai\-n\-der of $L $-series and $Z $-series.} \label {s3}
For a long time it is noticed (давным-давно известно), that the estimation (\ref {LeibOcenka}) in most cases gives very good accuracy. But, as the $L $-series which research differently as by means of a criterion (theorem) of Leibniz is impossible, its converge usually very slowly, and it would be desirable to have a method of specification of estimations (\ref {ocleib}) and (\ref {ocenka}).

\begin{exmp}\label{standart}
For a well-known series
\begin{equation}\label{garm}
    \sum\limits_{n=1}^{\infty} \frac{(-1)^{n+1}}{n}=\mbox{ln }2
\end{equation}
the estimation (\ref {ocleib}) gives an inequality $ |R_m | =\left |\sum\limits _ {n=m+1} ^ {\infty} \frac {(-1) ^ {n+1}} {n} \right |\leqslant R_m^L =\frac {1} {m+1} $.

Let's estimate $R_m $ more accurately.

$$
\begin{array}{c}
  |R_m|=\frac{1}{m+1}-\frac{1}{m+2}+\frac{1}{m+3}-\frac{1}{m+4}+\ldots= \\
  \\
  =\frac{1}{(m+1)(m+2)}+\frac{1}{(m+3)(m+4)}+\ldots= \\
  \\
  =\sum\limits_{k=1}^{\infty}\frac{1}{(m+2k-1)(m+2k)}; \\
  \\
  \int\limits_{1}^{\infty} \frac{1}{(m+2x-1)(m+2x)}d x<|R_m|<\int\limits_{0}^{\infty} \frac{1}{(m+2x-1)(m+2x)}d x;\\
  \\
   \frac{1}{2}\ln\left(1+\frac{1}{m+1}\right)<|R_m|<\frac{1}{2}\ln\left(1+\frac{1}{m-1}\right);\\
   \\
  |R_m|\sim \frac{1}{2m}\mbox{ if }m\rightarrow\infty.
\end{array}
$$
In this case the remainder of series monotonously converges to 0; the estimation error of (\ref {ocleib}) equal about one half of this estimation. \QEDA
\end{exmp}

\begin{exmp}
Let's consider another $L  $-series:
\begin{equation}\label{rd2}
\begin{array}{c}
    a_n=\left\{ \begin{array}{ccc}
                  \frac{1}{k}, & \mbox{ if } & n=2k-1; \\
                  \frac{1}{k}-\frac{1}{2^{k}}, & \mbox{ if }  & n=2k.
                \end{array}
    \right.\\ \\
    \mbox{or}\\ \\
    a_n=\frac{1}{\left[\displaystyle\frac{n+1}{2}\right]}-\displaystyle\frac{1+(-1)^n}{2^{\frac{n}{2}+1}}.
    \end{array}
\end{equation}

It is easy to see, that when $n  $ is big enough ($n> 7 \; $) then $ a_n\downarrow 0\,$ and
\begin{equation}\label{krivojleib}
    \sum\limits_{n=1}^{\infty}(-1)^{n+1}a_n=\frac{1}{1}-\left(\frac{1}{1}-\frac{1}{2}\right)+\frac{1}{2}-\left(\frac{1}{2}-\frac{1}{4}\right)+\frac{1}{3}-\left(\frac{1}{3}-\frac{1}{8}\right)+\ldots=1.
\end{equation}

Thus
\begin{equation}\label{ocRn}
    R_n=\left\{ \begin{array}{ccc}
                  \frac{1}{2^k}-\frac{1}{k}, & \mbox{ if } & n=2k-1; \\
                  \frac{1}{2^k}, & \mbox{ if }  & n=2k.
                \end{array}
    \right.\\
\end{equation}

But the estimation of the remainder of a series (\ref {krivojleib}) according to (\ref {ocleib}) is that:
$$
\begin{array}{c}
|R_{2n-1}|\leqslant R_{2n-1}^L=\frac{2}{n+1}-\frac{1}{2^{\frac{n+1}{2}}},\\
  |R_{2n}|\leqslant R_{2n}^L=\frac{1}{n+1}.
\end{array}
$$
That is
$$
\begin{array}{c}
  \lim\limits_{n\rightarrow\infty} \frac{|R_{2\omega }|}{R_{2\omega }^L}=0, \\ \\
  \lim\limits_{n\rightarrow\infty} \frac{|R_{2\omega +1}|}{R_{2\omega +1}^L}=1;
\end{array}
$$
 and accuracy of an estimation has a big fluctuations. The convenient ratio similar to  $ |R_n-R_n^L |\lesssim C\cdot R_n $ here is not present. In this case it is possible to speak about unsatisfactory accuracy of an estimation (\ref {ocleib}). The presented case has some similarity to an example (\ref {3-z}). \QEDA
\end{exmp}

\begin{theorem}\label{xvosty}
If the sequence $ \{a_n \} $ monotonously decreases  to $0$ $ (a_n\downarrow 0) $ at $n> n_0$, and at $n> n_0$ the condition $a _ {n+1} \leqslant \frac {a _ {n} +a _ {n+2}} {2} $ is satisfied at $m> n_0$ then the estimation $R_m^L $ the remainder $L a $-series
$ \sum\limits _ {k=1} ^ {\infty} (-1) ^na_n $  is comparable by absolute value of the remainder $R_m $ this series: $ \frac {1} {2} R_m^L\leqslant |R_m |\leqslant R_m^L $. \QEDB
\end{theorem}

\begin{proof}
As $a _ {n+1} \leqslant \frac {a _ {n} +a _ {n+2}} {2} $, it is possible to find such twice differentiable function $f (x) $, convex downwards at $x> n_0 \; $ ($ f'' (x)\geqslant 0$ at $x> n_0$), that $a_n=f (n) $.

The series remainder $ \sum\limits _ {k=1} ^ {\infty} (-1) ^nf (n) $ can be estimated as follows:

$$
\begin{array}{c}
  |R_n|=\left|\sum\limits_{k=n}^{\infty} (-1)^nf(n)\right|= \\ \\
  =\left((f(n)-f(n+1))+(f(n+2)-f(n+3))+(f(n+4)-f(n+5))+\ldots\right)= \\ \\
  =\left(f'(\xi_n) +f'(\xi_{n+2})+f'(\xi_{n+4})+\ldots\right)\leqslant\\ \\
  \leqslant \left(f'(n)+f'(n+2)+f'(n+4)+\ldots\right);\\ \\
  |R_n|\geqslant \left(f'(n+1)+f'(n+3)+f'(n+5)+\ldots\right),
\end{array}
$$
here $\xi_n\in [n;n+1]$; therefore
$$
  |R_n|\leqslant\int\limits_0^{\infty}f'(n+2x)dx=\frac{1}{2}\int\limits_n^{\infty}f'(n+2x)d2x=\frac{1}{2}f(n).
$$

It is similarly possible to receive an estimation from below:
$$
|R_n|\geqslant \frac{1}{2}f(n+1).
$$

Thus,
\begin{equation}\label{superoc}
   \frac{1}{2}a_{n+1}\leqslant|R_n|\leqslant\frac{1}{2}a_{n}\mbox{ if }n>n_0.
\end{equation}

The latter inequality is important because the Leibniz Theorem \ref {t2} is usually applied to series with slowly decreasing components, and in this case the in\-e\-qua\-li\-ty $ |R_n |\leqslant\frac {1} {2} a _ {n} $ is stronger than the inequality $ |R_n |\leqslant
a_{n+1}$.
\QEDA
\end{proof}

\begin{Corollary}
For a $Z (p) $-series ($p=2\omega-1$) $ \sum\limits _ {n=n_0} ^ \infty (-1) ^n a_n $ it is possible to give the following general estimation of the remainder of a series:
$$
R_m=(-1)^m(\delta_1-\delta_2+\delta_3-\ldots-\delta_{p-1}+\delta_p),
$$
where $\delta _k$ is a remainder of $L$-series. Thus
$\frac 12 a_{n+i}\leqslant\delta _i\leqslant a_{n+i}$. Therefore
$$
\begin{array}{c}
  |R_m|\leqslant a_{m+1}-\frac{1}{2}a_{m+2}+a_{m+3}-\frac{1}{2}a_{m+4}+\ldots-\frac{1}{2}a_{m+p-1}+a_{m+p}\leqslant\\ \\
  \leqslant\max( a_i,\;i=m+1,m+3,\ldots,m+p)\cdot\frac{p+1}{2}-\\ \\
  -\frac{1}{2}\min (a_i,\;i=m+2, m+4,\ldots,m+p-1)\cdot\frac{p-1}{2},
\end{array}
$$
Besides,
$$
\begin{array}{c}
 |R_m|\geqslant\frac{1}{2}a_{m+1}-a_{m+2}+\frac{1}{2}a_{m+3}-a_{m+3}+\ldots+\frac{1}{2}a_{m+p}\geqslant\\ \\
  \geqslant\frac{1}{2}\min ( a_i,\;i=m+1,m+3,\ldots,m+2\omega -1)\cdot\frac{p+1}{2}-\\ \\
   -\max{(a_i,\;i=m+2, m+2,\ldots,m+2\omega -2)}\cdot\frac{p-1}{2}.
\end{array}
$$

The last estimate, most likely, is uninteresting: the right hand side of an in\-equa\-li\-ty will be almost always negative.

However if at $n> n_0$ the inequality $a_n\leqslant 2a _ {n+p} $ is true\footnote {It denotes, that members of each $L $-series composing a $Z  $-series decrease more slowly than a geometrical progression with a denominator 0,5. Considering, that the Leibniz Theorem \ref {t2} is applied, basically, to conditionally (and very slowly) converging  series, such assumption is pertinent.
} then the estimation $R_m $ can be improved, considering, that in this case
$$
\begin{array}{c}
  \frac 12 a_{n+i}\leqslant\delta _i \leqslant \frac 12 a_{n+i-p}\leqslant a_{n+i}:\\ \\
|R_m|\leqslant \frac{1}{2}(a_{m+1-p}-a_{m+2}+a_{m+3-p}-a_{m+4}+\ldots-a_{m+p-1}+a_{m})\leqslant\\ \\
  \leqslant \frac 12 \left(\max( a_{i-p},\;i=m+1,m+3,\ldots,m+p)\cdot\frac{p+1}{2}-\right.\\ \\
  -\left. \min (a_i,\;i=m+2, m+4,\ldots,m+p-1)\cdot\frac{p-1}{2}\right).
\end{array}
$$

If, in addition to these conditions, for any fixed $k $ at $n\rightarrow\infty $ the condition \linebreak $a_n\sim a _ {n+k} $ is correct, it is possible to assert, that when $m\rightarrow \infty $ then\linebreak $R_m\lesssim \frac 12 a_m $. \QEDA
\end{Corollary}

This situation take place in the Example (\ref {5-z}): a series $ \sum\limits _ {n=1} ^ {\infty} \frac {(-1) ^ {(n-1)}} {n+2\cos n} $ con\-ver\-ges, and its remainder $R_m\lesssim \frac 12 \cdot \frac 1m $.

\section{Some remarks}

The theorem \ref{t1} can be generalised by different natural way. For example, let denote the series $\sum\limits_{k=0}^{\infty} a_k$ ($a_k\neq 0\;\forall k\in \mathbb{N}$) as a $\omega$-periodical-single series if for some $\omega\in \mathbb{N}$ it is carried out $\forall k>k_0 \;\mbox{sign }(a_k)=-\mbox{sign }(a_{k+\omega})$, here $\mbox{sign }(a_k)=1$ if $a_k>0$ and $\mbox{sign }(a_k)=-1$ if $a_k>0$. Then if consequence $\{|a_k|\}$  is a $Z(\omega)$-monotonously decreasing to zero, then the series $\sum\limits_{k=0}^{\infty} a_k$ converges. The estimates for the remainder of this series can be found by the reasonings above. This example include the situation when some subseries ($L$-series) of $Z$-series are zero. In other words, the $\omega$-periodical-single series can be transformed to the $Z(2n-1)$-series by addition of some quantity of zero-series.

Author suppose the using of notion of $Z$-series in some domains of theory of series, for example, in the questions of some transformations of the divergent series.

~

~

~

~\\
\emph{Author's information}\\
Galina A.Zverkina\\
Department of Applied Mathematics, \\ Moscow State University of Railway
Engineering (MIIT), Russia\\
\emph{E-mail: zvеrkinа @ gmаil.соm}

\end{document}